\documentclass{amsart}

\usepackage{amscd,amsmath,amssymb,amsfonts,bbm, calligra, xspace}
\usepackage[T1]{fontenc}
\usepackage{lmodern}
\usepackage{mathtools}
\usepackage{enumerate}
\usepackage{multicol}
\usepackage[all]{xy}
\usepackage{mathrsfs}
\usepackage{footmisc}
\usepackage[unicode,pdfborder={0 0 0},final]{hyperref}

\newtheorem{thm}{Theorem}[section]
\newtheorem{prop}[thm]{Proposition}

\newtheorem{lem}[thm]{Lemma}

\theoremstyle{definition}

\theoremstyle{remark}

\numberwithin{equation}{section}

\newcommand{\Ree}{\mathrm{Re}}
\newcommand{\Ima}{\mathrm{Im}}

\newcommand{\N}{\mathrm{N}}

\newcommand{\red}{\mathrm{red}}

\newcommand{\Spec}{\mathrm{Spec}}

\newcommand{\isoto}{\myxrightarrow{\,\sim\,}}
\makeatletter
\def\myrightarrow{{\setbox\z@\hbox{$\rightarrow$}\dimen0\ht\z@\multiply\dimen0 6\divide\dimen0 10\ht\z@\dimen0\box\z@}}
\def\myrightarrowfill@{\arrowfill@\relbar\relbar\myrightarrow}
\newcommand{\myxrightarrow}[2][]{\ext@arrow 0359\myrightarrowfill@{#1}{#2}}
\makeatother

\def\loccit{\emph{loc}.\kern3pt \emph{cit}.{}\xspace}
\def\eg{e.g.\kern.3em}

\def\resp {\text{resp.}\kern.3em}

\def\C{\mathbb C}

\def\R{\mathbb R}
\def\N{\mathbb N}

\def\cO{\mathcal{O}}

\def\cF{\mathcal{F}}

\def\cS{\mathcal{S}}

\def\km{\mathfrak{m}}

\begin{document}

\title[]{Stein spaces and Stein algebras}

\author{Olivier Benoist}
\address{D\'epartement de math\'ematiques et applications, \'Ecole normale sup\'erieure, CNRS,
45 rue d'Ulm, 75230 Paris Cedex 05, France}
\email{olivier.benoist@ens.fr}

\renewcommand{\abstractname}{Abstract}
\begin{abstract}
We prove that the category of Stein spaces and holomorphic maps is anti-equivalent to the category of Stein algebras and $\C$-algebra morphisms. This removes a finite dimensionality hypothesis from a theorem of Forster.
\end{abstract}
\maketitle

\section*{Introduction}

Complex spaces are a generalization of complex manifolds allowing singularities, and as such are the basic objects of study in complex-analytic geometry. 
Formally, they are defined to be $\C$-ringed spaces that are locally isomorphic to model spaces defined by the vanishing of finitely many holomorphic functions in a domain of~$\C^N$ for some $N\geq 0$ (see \cite[1, \S 1.5]{GRCoherent}). 
We assume that they are second\nobreakdash-coun\-ta\-ble, but not necessarily reduced or finite-dimensional.

A complex space~$S$ is said to be \textit{Stein} if $H^k(S,\cF)=0$ for all coherent sheaves $\cF$ on~$S$ and all $k>0$ (see \cite{GRStein}). Stein spaces are the complex-analytic analogues of affine algebraic varieties.  For instance, the Stein spaces of finite embedding dimension are exactly those complex spaces that may be realized as closed complex subspaces of $\C^N$ for some $N\geq 0$ (see \cite[Theorem 6]{Narasimhan}). 

If $S$ is a complex space, the $\C$-algebra $\cO(S)$ of holomorphic functions on $S$ carries a canonical Fr\'echet topology (see \cite[V, \S 6]{GRStein}).  A topological $\C$-algebra of the form $\cO(S)$ for some Stein space $S$ is called a \textit{Stein algebra}. 

In algebraic geometry, the anti-equivalence of categories between affine varieties over~$\C$ and $\C$-algebras of finite type is a basic tool to study affine algebraic varieties. Our main theorem is a counterpart of this result in complex-analytic geometry.

\begin{thm}[Theorem \ref{th}]
\label{thmain}
The contravariant functor 
\begin{equation}
\label{antieq}
 \left\{  \begin{array}{l}
    \textrm{\hspace{2.2em}Stein spaces}\\
\textrm{and holomorphic maps}
  \end{array}\right\}\to
 \left\{  \begin{array}{l}
    \textrm{\hspace{2.5em}Stein algebras}\\
\textrm{and $\C$-algebra morphisms}
  \end{array}\right\}
\end{equation}
given by $S\mapsto \cO(S)$  is an anti-equivalence of categories.
\end{thm}

Very significant particular cases of Theorem \ref{thmain} were previously known.  
First, Forster has shown in \cite[Satz 1]{Forster} that Theorem \ref{thmain} holds if one replaces  the right-hand side of (\ref{antieq}) by the category of Stein algebras and continuous $\C$-algebra morphisms. From this point of view,  our contribution is an automatic continuity result for morphisms of Stein algebras (see Theorem \ref{thcont} below).

 Second, Forster has proven this automatic continuity result in restriction to finite-dimensional Stein spaces (see \cite[Theorem 5]{ForsterUniqueness}). 
 In particular, Theorem~\ref{thmain} was already known in restriction to finite-dimensional Stein spaces and their associated Stein algebras.
Forster's theorem was later generalized by 
Markoe \cite{MarkoeStein} and 
Ephraim \cite[Theorem 2.3]{Ephraim} who made weaker finite dimensionality assumptions.  Our contribution is to remove these finite dimensionality hypotheses altogether. This problem was raised by Forster in \cite[Remark p.162]{ForsterUniqueness}.

Our strategy to prove Theorem \ref{thmain} is to reduce to the finite-dimensional case treated by Forster by means of the next theorem.

\begin{thm}[Theorem \ref{prop}]
\label{propmain}
Let $S$ be a Stein space. Then there exists a holomorphic map ${f:S\to\C^2}$ all of whose fibers are finite-dimensional.
\end{thm}

Our proof of Theorem \ref{propmain} is an application of Oka theory. It uses in a crucial way new examples of Oka manifolds constructed by Forstneri\v{c} and Wold \cite{FW} (based on and extending earlier work of Kusakabe~\cite{Kusakabe1, Kusakabe}), as well as an extension theorem for holomorphic maps from Stein spaces to Oka manifolds due to Forstneri\v{c} \cite{Forstextension, Forstneric}. 

We note that Theorem \ref{propmain} is optimal in the sense that there may not exist a holomorphic map $f:S\to\C$ with finite-dimensional fibers (see Proposition~\ref{rem}). An earlier version of this article,  relying on the Oka manifolds constructed by Kusakabe \cite[Theorem 1.6]{Kusakabe}, only produced such a map with values in $\C^3$. We are grateful to Franc Forstneri\v{c} for drawing our attention to the article \cite{FW}, thereby allowing us to prove Theorem \ref{propmain} in the form stated above.

The results of Oka theory that we need are gathered in Section \ref{sec1}.
These tools are used to prove Theorem \ref{propmain} in Section \ref{sec2}.  In Section \ref{sec3}, we deduce Theorem~\ref{thmain} from Theorem~\ref{propmain} and from Forster's works \cite{ForsterUniqueness, Forster}.

\section{Tools from Oka theory}
\label{sec1}

We recall that a complex manifold $Y$ is said to be \textit{Oka} if for all convex compact subsets $K\subset\C^N$ and all open neighborhoods $\Omega$ of $K$ in $\C^N$, any holomorphic map~$\Omega\to Y$ can be approximated uniformly on $K$ by holomorphic maps $\C^N\to Y$ (see \cite[Definition 1.2]{ForstOka}).

We now introduce the Oka manifolds of interest to us. 
For $r\in\R$, define 
$$Y_r:=\{(z_1,z_2)\in\C^2\mid \Ima(z_2)<|z_1|^2+\Ree(z_2)^2-r\}.$$
The next proposition is a particular case of a theorem of Forstneri\v{c} and Wold \cite[Corollary 1.5]{FW} (pointed out in \cite[(1.2)]{FW}).

\begin{prop}
\label{propFW}
For $r\in\R$, the complex manifold $Y_r$ is Oka.
\end{prop}

The following easy lemma implies in particular that $Y_r$ is contractible.

\begin{lem}
\label{lemhomo}
Fix $r\in\R$. There is a homotopy $(h_t)_{t\in [0,1]}:\C^2\to \C^2$ inducing strong deformation retractions of both $\C^2$ and $Y_r$ onto $\{(z_1,z_2)\in\C^2\mid\Ima(z_2)\leq -r-1\}$.
\end{lem}

\begin{proof}
The homotopy $(h_t)_{t\in[0,1]}$ defined by
\phantom{\qedhere}
\begin{equation*}
\begin{alignedat}{5}
\label{eqhomo}
&h_t(z_1,z_2)=(z_1,z_2-it(\Ima(z_2)+r+1))&&\textrm{ \hspace{1em}if } \Ima(z_2)\geq -r-1\\
&h_t(z_1,z_2)=(z_1,z_2)&&\textrm{ \hspace{1em}if } \Ima(z_2)\leq -r-1
\end{alignedat}
\end{equation*}
has the required properties.
\end{proof}

We will make use of the Oka property and of the contractibility of $Y_r$ through the next extension result, which is an application of theorems of Forstneri\v{c} (see~\cite[Theorem 1.1]{Forstextension} and the more general \cite[Theorem 5.4.4]{Forstneric}). 

\begin{prop}
\label{Okaprop}
Fix $r\in\R$. 
Let $S$ be a reduced Stein space and let~$S'$ be a (possibly nonreduced) closed complex subspace of~$S$. Let $f':S'\to Y_r$ be a holomorphic map. Then there exists a holomorphic map $f:S\to Y_r$ with $f|_{S'}=f'$.
\end{prop}

\begin{proof}
Since $S$ is Stein,  the restriction map $\cO(S)\to\cO(S')$ is onto. It follows that there exists a holomorphic map $f_1:S\to\C^2$ such that $f_1|_{S'}=f'$.  

Define $U:=f_1^{-1}(Y_r)$. It is an open neighborhood of $S'$ in $S$.  Let $Z\subset U$ be a closed neighborhood of $S'$ in $U$. By the Tietze--Urysohn extension theorem, there exists a continuous map $\tau:S\to[0,1]$ which is equal to $0$ on $Z$ and to $1$ on $S\setminus U$. 

Define a continuous map $f_2:S\to Y_r$ by the formula $f_2(s)=h_{\tau(s)}(f_1(s))$, where~$(h_t)_{t\in[0,1]}$ is the homotopy given by Lemma \ref{lemhomo}. Since $f_2$ is equal to $f_1$ on~$U$,  it is holomorphic in a neighborhood of $S'$ and satisfies $f_2|_{S'}=f'$.

As $Y_r$ is Oka by Proposition \ref{propFW}, it now follows from the jet interpolation part of \cite[Theorem~5.4.4]{Forstneric} (applied with $\pi$ equal to be the first projection map $S\times Y\to S$ and with~$\cS$ equal to the ideal sheaf of $S'$ in $S$) that $f_2$ is homotopic to a holomorphic map $f:S\to Y_r$ with~$f|_{S'}=f_2|_{S'}$, and hence $f|_{S'}=f'$. This completes the proof of the proposition.
\end{proof}

\section{Holomorphic maps with finite-dimensional fibers}
\label{sec2}

The next theorem is the key to our main results.

\begin{thm}
\label{prop}
Let $S$ be a Stein space. Then there exists a holomorphic map ${f:S\to\C^2}$ all of whose fibers are finite-dimensional. 
\end{thm}

\begin{proof}
Let $S^{\red}$ be the reduction of $S$. Since $S$ is Stein, the restriction map ${\cO(S)\to\cO(S^{\red})}$ is onto,  and we may assume that $S$ is reduced.

Let $(S_k)_{0\leq k < n}$ with $n\in\N\cup\{+\infty\}$ be the irreducible components of $S$,  viewed as reduced closed complex subspaces of $S$. 
Let $\Theta$ be the collection of all reduced and irreducible closed complex subspaces of $S$ that may be obtained as irreducible components of an intersection of finitely many of the $S_k$. The set $\Theta$ is at most countable, and any compact subset of $S$ meets at most finitely many elements of~$\Theta$.  

For $d\geq 0$, we let $\Theta_d\subset\Theta$ be the set of all $d$-dimensional elements of $\Theta$.  Let~$(Z_{d,j})_{0\leq j< m(d)}$ with $m(d)\in\N\cup\{+\infty\}$ be an enumeration of the elements of~$\Theta_d$. 
We henceforth identify $\Theta$ with the set of all pairs $(d,j)$ with $d\geq 0$ and $0\leq j< m(d)$ and endow it with the lexicographical order. It is a well-ordered set.
For all~$(d,j)\in\Theta$,  we view $W_{d,j}:=\cup_{(d',j')\leq(d,j)}Z_{d',j'}$ and $W'_{d,j}:=\cup_{(d',j')<(d,j)}Z_{d',j'}$ as reduced closed complex subspaces of $S$. Finally, for $(d,j)\in\Theta$, we let $r(d,j)$ be the biggest integer~$k\geq 1$ such that $Z_{d,j}\subset S_k$.

We will now construct holomorphic functions $f_{d,j}:W_{d,j}\to \C^2$ for all ${(d,j)\in \Theta}$ with the property that $f_{d,j}|_{W_{d',j'}}=f_{d',j'}$ and $f_{d,j}(Z_{d',j'})\subset Y_{r(d',j')}$ whenever $(d',j')\leq (d,j)$. The construction is by induction on the pair $(d,j)\in\Theta$ (which is legitimate since $\Theta$ is well-ordered).

Assume that the $f_{d',j'}$ for $(d',j')<(d,j)$ have been constructed.  Since these maps are compatible,  they glue to give rise to a holomorphic map $f'_{d,j}:W'_{d,j}\to \C^2$.  Now $W_{d,j}=W'_{d,j}\cup Z_{d,j}$.  Define $V_{d,j}:=W'_{d,j}\cap Z_{d,j}$. It is a possibly nonreduced closed complex subspace of $S$.  Note that $V_{d,j}$ is set-theoretically a union of some of the~$Z_{d',j'}$ with $(d',j')<(d,j)$.  If $Z_{d',j'}\subset V_{d,j}$ is one of them, then~$Z_{d',j'}\subset Z_{d,j}$ and hence $r(d',j')\geq r(d,j)$. Since $f_{d',j'}(Z_{d',j'})\subset Y_{r(d',j')}\subset Y_{r(d,j)}$, we deduce that~$f'_{d,j}(V_{d,j})\subset Y_{r(d,j)}$. 
Proposition \ref{Okaprop} now implies that the holomorphic map~${f'_{d,j}|_{V_{d,j}}:V_{d,j}\to Y_{r(d,j)}}$ extends to a holomorphic map~$f''_{d,j}:Z_{d,j}\to Y_{r(d,j)}$. Since~$f'_{d,j}$ and~$f''_{d,j}$ coincide on $V_{d,j}=W'_{d,j}\cap Z_{d,j}$, they glue (by Lemma \ref{lemglue} below) to give rise to a holomorphic map~$f_{d,j}:W_{d,j}\to \C^2$ with the required properties.

As the $(f_{d,j})_{(d,j)\in\Theta}$ are compatible, they induce a holomorphic map $f:S\to \C^2$.  Let us verify that this map has the required property. 
One has $f(S_k)\subset Y_k$ for all~$0\leq k< n$ (as $S_k$ is one of the $Z_{d,j}$). Since the $(Y_k)_{k\geq 0}$ form a decreasing family of subsets of $\C^2$ with empty intersection, we deduce that any point of $\C^2$ belongs to at most finitely many of the $f(S_k)$.  In other words, any fiber of $f$ intersects at most finitely many of the $S_k$.  It follows that all the fibers of $f$ are finite-dimensional.
\end{proof}

\begin{lem}
\label{lemglue}
Let $S$ be a complex space. Let $S_1$ and $S_2$ be closed complex subspaces of~$S$. Set $T:=S_1\cap S_2$. The following diagram of sheaves on~$S$ is~exact:
\begin{equation}
\label{exact}
\cO_S\xrightarrow{f\mapsto (f|_{S_1},f|_{S_2})}\cO_{S_1}\oplus\cO_{S_2}\xrightarrow{(g,h)\mapsto g|_{T}-h|_{T}} \cO_{T}\to 0.
\end{equation}
If moreover $S$ is reduced and $S=S_1\cup S_2$, then the left arrow of (\ref{exact}) is injective.
\end{lem}

\begin{proof}
Fix $s\in S$.  Write $A=\cO_{S,s}$ and let $I_1$ (\resp $I_2$) be the ideal of $A$ consisting of germs of functions vanishing on $S_1$ (\resp $S_2$). Then the exactness of (\ref{exact}) at $s$ results from the exactness of 
$A\to A/I_1\oplus A/I_2\to A/\langle I_1,I_2\rangle\to 0$,
which is valid for any two ideals $I_1$ and $I_2$ of a commutative ring $A$.

If $S=S_1\cup S_2$, then a holomorphic function in the kernel of the left arrow of~(\ref{exact}) vanishes at all points and hence vanishes if $S$ is reduced.
\end{proof}

The next proposition shows the optimality of Theorem \ref{prop}.

\begin{prop}
\label{rem}
There exists a Stein space $S$ such that all holomorphic maps $f:S\to\C$ admit an infinite-dimensional fiber.
\end{prop}

\begin{proof}
For $n\geq 1$, set $S_n:=\C^n$.  Define $T_n:=\{(z_1,\dots,z_n)\in S_n\mid z_n=0\}$ and $T'_n:=\{(z_1,\dots,z_{n+1})\in S_{n+1}\mid z_n=0\textrm{ and }z_{n+1}=1\}$. Let $\varphi_n:T_n\isoto T'_n$  be the isomorphism given by $\varphi_n(z_1,\dots,z_{n-1},0)=(z_1,\dots,z_{n-1},0,1)$.
Let $S$ be the complex space obtained from $\sqcup_{n\geq 1}S_n$ by gluing $S_n$ and $S_{n+1}$ transversally 
along~$T_n$ and $T_n'$ by means of $\varphi_n$ (for all $n\geq 1$).
The complex space $S$ is Stein because so is its normalization~$\sqcup_{n\geq 1}S_n$ (see~\cite[Theorem 1]{Narasimhannormalization}).

Let $f:S\to\C$ be a holomorphic map. Assume first that $f|_{S_n}$ is constant for all~$n\gg 0$. As the subset  $S_n\cap S_{n+1}$ of $S$ is nonempty, the value taken by $f|_{S_n}$ does not depend on $n\gg0$. It follows that $f$ has a (single) infinite-dimensional fiber. 

Assume now that the set $\Sigma:=\{n\in\N_{\geq 1}\mid f|_{S_n}\textrm{ is not constant}\}$ is infinite.  For all~$n\in\Sigma$, the map $f|_{S_n}:S_n\to\C$ omits at most one value, by Picard's little theorem.  
We deduce that at most one complex number is not the image of $f|_{S_n}$ for all but finitely many $n\in\Sigma$.  Consequently, all complex numbers except possibly one  are in the image of infinitely of the $f|_{S_n}$.  As the nonempty fibers of $f|_{S_n}$ have dimension~$\geq n-1$, we deduce that all the fibers of $f$ except possibly one are infinite-dimensional.
\end{proof}

\section{Morphisms of Stein algebras}
\label{sec3}

\begin{thm}
\label{characters}
Let $S$ be a Stein space.  Let $\chi:\cO(S)\to \C$ be a $\C$\nobreakdash-al\-ge\-bra morphism. Then $\chi$ is continuous and there exists $s\in S$ such that $\chi(f)=f(s)$ for all $f\in\cO(S)$.
\end{thm}

\begin{proof}
Let $f:S\to\C^2$ be as in Theorem \ref{prop}. Let $(f_i)_{1\leq i\leq 2}$ be the components of~$f$.  Set $\lambda_i:=\chi(f_i)\in\C$.  Let $T\subset S$ be the closed complex subspace defined by the equations $\{f_i=\lambda_i\}_{1\leq i\leq 2}$.  
Let $r_{S,T}:\cO(S)\to\cO(T)$ be the restriction map, which is continuous by \cite[V, \S 6.4 Theorem 6]{GRStein}.
By \cite[Lemma 1.7]{Ephraim},  there exists a morphism of $\C$-algebras $\chi_T:\cO(T)\to \C$ such that $\chi=\chi_T\circ r_{S,T}$.

 Our choice of $f$ implies that $T$ is a finite-dimensional Stein space. It therefore follows from Forster's theorem \cite[Theorem 5]{ForsterUniqueness} that $\chi_T$ is continuous and hence that so is $\chi$. 
 Another theorem of Forster \cite[Satz 1]{Forster} then implies that there exists $s\in S$ such that $\chi(f)=f(s)$ for all $f\in\cO(S)$.
\end{proof}

\begin{thm}
\label{thcont}
Any $\C$-algebra morphism between Stein algebras is continuous.
\end{thm}

\begin{proof}
Let $S$ and $S'$ be two Stein spaces, and let $\xi:\cO(S')\to\cO(S)$ be a $\C$\nobreakdash-algebra morphism. 
Fix a finitely generated maximal ideal  $\km\subset \cO(S)$.
There exists $s\in S$ such that $\km=\{f\in\cO(S)\mid f(s)=0\}$ (see \eg  \cite[V, \S 7.1,  statement above Theorem 1]{GRStein}). Evaluation at $s$ therefore induces an isomorphism $\cO(S)/\km\isoto\C$. 
We let $\chi:\cO(S)\to\C$ be the induced map.

Apply Theorem \ref{characters} to the $\C$-algebra morphism $\chi\circ\xi:\cO(S')\to \C$.  We deduce the existence of $s'\in S'$ such that $\chi\circ\xi(f)=f(s')$ for all $f\in\cO(S')$.  It then follows that $\xi^{-1}(\km)=\{f\in\cO(S')\mid f(s')=0\}$. This maximal ideal is closed (by continuity of the evaluation map $f\mapsto f(s')$), and hence finitely generated by \cite[Theorem~2]{Forster}. 

Since $\km$ was arbitrary, the continuity of $\xi$ is now an application of the criterion given in \cite[Theorem 3]{Forster}.
\end{proof}

\begin{thm}
\label{th}
The contravariant functor 
\begin{equation}
\label{antieq}
 \left\{  \begin{array}{l}
    \textrm{\hspace{2.2em}Stein spaces}\\
\textrm{and holomorphic maps}
  \end{array}\right\}\to
 \left\{  \begin{array}{l}
    \textrm{\hspace{2.5em}Stein algebras}\\
\textrm{and $\C$-algebra morphisms}
  \end{array}\right\}
\end{equation}
given by $S\mapsto \cO(S)$  is an anti-equivalence of categories.
\end{thm}

\begin{proof}
Since $\C$-algebra morphisms of Stein algebras are automatically continuous by Theorem \ref{thcont},  the theorem is equivalent to \cite[Satz 1]{Forster}.
\end{proof}

We finally record the following consequence of Theorem \ref{th} for use in \cite{Steinsurface}. If~$S$ is a Stein space, we let $\lambda_S:S\to \Spec(\cO(S))$ be the unique morphism of locally ringed spaces such that $\lambda_S^*:\cO(S)\to\cO(S)$ is the identity (see \cite[Lemma 01I1]{SP}).

\begin{prop}
\label{propBing}
Let $X$ be a complex space and let $S$ be a Stein space.  The map
\begin{equation}
\label{bij}
 \left\{  \begin{array}{l}
    \textrm{holomorphic maps}\\
\textrm{\hspace{2.4em}$X\to S$}
  \end{array}\right\}\to
 \left\{  \begin{array}{l}
    \textrm{morphisms of $\C$-locally ringed spaces}\\
\textrm{\hspace{3.9em}$X\to \Spec(\cO(S))$}
  \end{array}\right\}
\end{equation}
given by $f\mapsto \lambda_S\circ f$ is a bijection.
\end{prop}

\begin{proof}
As the statement is local on $X$, we may assume that $X$ is Stein. In this case, the proposition follows from Theorem \ref{th} since the global sections functor induces a bijection between the set of morphisms of $\C$-locally ringed spaces $X\to\Spec(\cO(S))$ and the set of $\C$-algebra morphisms $\cO(S)\to\cO(X)$ (see \cite[Lemma 01I1]{SP}). 
\end{proof}

\bibliographystyle{myamsalpha}
\bibliography{Steinalgebras}

\providecommand{\bysame}{\leavevmode\hbox to3em{\hrulefill}\thinspace}
\providecommand{\MR}{\relax\ifhmode\unskip\space\fi MR }
\providecommand{\MRhref}[2]{%
  \href{http://www.ams.org/mathscinet-getitem?mr=#1}{#2}
}
\providecommand{\href}[2]{#2}
\begin{thebibliography}{Mar73}

\bibitem[Ben24]{Steinsurface}
O.~Benoist, \emph{On the field of meromorphic functions on a {Stein} surface},
  preprint 2024, \href{http://arxiv.org/abs/2410.16809}{arXiv:2410.16809}, to
  appear in Crelle’s Journal, special bicentennial issue 2026.

\bibitem[Eph78]{Ephraim}
R.~Ephraim, \emph{Multiplicative linear functionals of {S}tein algebras},
  Pacific J. Math. \textbf{78} (1978), no.~1, 89--93.

\bibitem[For66]{ForsterUniqueness}
O.~Forster, \emph{Uniqueness of topology in {S}tein algebras}, Function
  {A}lgebras ({P}roc. {I}nternat. {S}ympos. on {F}unction {A}lgebras, {T}ulane
  {U}niv., 1965), Scott-Foresman, Chicago, Ill., 1966, pp.~157--163.

\bibitem[For67]{Forster}
\bysame, \emph{Zur {T}heorie der {S}teinschen {A}lgebren und {M}oduln}, Math.
  Z. \textbf{97} (1967), 376--405.

\bibitem[For05]{Forstextension}
F.~Forstneri\v{c}, \emph{Extending holomorphic mappings from subvarieties in
  {S}tein manifolds}, Ann. Inst. Fourier \textbf{55} (2005), no.~3, 733--751.

\bibitem[For09]{ForstOka}
\bysame, \emph{Oka manifolds}, C. R. Math. Acad. Sci. Paris \textbf{347}
  (2009), no.~17-18, 1017--1020.

\bibitem[For17]{Forstneric}
\bysame, \emph{Stein manifolds and holomorphic mappings}, second ed., Ergeb.
  Math. Grenzgeb. (3), vol.~56, Springer, Cham, 2017.

\bibitem[FW24]{FW}
F.~Forstneri\v{c} and E.~F. Wold, \emph{Oka domains in {E}uclidean spaces},
  IMRN (2024), no.~3, 1801--1824.

\bibitem[GR79]{GRStein}
H.~Grauert and R.~Remmert, \emph{Theory of {S}tein spaces}, Grundlehren der
  Mathematischen Wissenschaften, vol. 236, Springer-Verlag, Berlin-New York,
  1979.

\bibitem[GR84]{GRCoherent}
\bysame, \emph{Coherent analytic sheaves}, Grundlehren der Mathematischen
  Wissenschaften, vol. 265, Springer-Verlag, Berlin, 1984.

\bibitem[Kus21]{Kusakabe1}
Y.~Kusakabe, \emph{Elliptic characterization and localization of {O}ka
  manifolds}, Indiana Univ. Math. J. \textbf{70} (2021), no.~3, 1039--1054.

\bibitem[Kus24]{Kusakabe}
\bysame, \emph{Oka properties of complements of holomorphically convex sets},
  Ann. of Math. (2) \textbf{199} (2024), no.~2, 899--917.

\bibitem[Mar73]{MarkoeStein}
A.~Markoe, \emph{Maximal ideals of {S}tein algebras}, Conference on Complex
  Analysis, State University of New York at Buffalo, 1973, pp.~25--39.

\bibitem[Nar60]{Narasimhan}
R.~Narasimhan, \emph{Imbedding of holomorphically complete complex spaces},
  Amer. J. Math. \textbf{82} (1960), 917--934.

\bibitem[Nar62]{Narasimhannormalization}
\bysame, \emph{A note on {S}tein spaces and their normalisations}, Ann. Scuola
  Norm. Sup. Pisa Cl. Sci. (3) \textbf{16} (1962), 327--333.

\bibitem[SP]{SP}
A.~J. de~Jong et~al., \emph{\textit{The Stacks Project}},
  \url{https://stacks.math.columbia.edu}.

\end{thebibliography}

\end{document}